\documentclass[11pt]{amsart}
\def\@citecolor{blue}
\def\@urlcolor{blue}
\def\@linkcolor{blue}
\makeatother
\usepackage[usenames,dvipsnames]{xcolor}
\makeatletter
\newcommand{\inlineitem}[1][]{%
\ifnum\enit@type=\tw@
    {\descriptionlabel{#1}}
  \hspace{\labelsep}
\else
  \ifnum\enit@type=\z@
       \refstepcounter{\@listctr}\fi
    \quad\@itemlabel\hspace{\labelsep}
\fi}
\makeatother

\usepackage{hyperref}
\hypersetup{
    bookmarks=true,         
    unicode=false,          
    pdftoolbar=true,        
    pdfmenubar=true,        
    pdffitwindow=false,     
    pdfstartview={FitH},    
    pdftitle={My title},    
    pdfauthor={Author},     
    pdfsubject={Subject},   
    pdfcreator={Creator},   
    pdfproducer={Producer}, 
    pdfkeywords={keyword1} {key2} {key3}, 
    pdfnewwindow=true,      
    colorlinks=false,       
    linkcolor=red,          
    citecolor=green,        
    filecolor=magenta,      
    urlcolor=cyan           
}

\usepackage[headings]{fullpage}
\usepackage{lmodern}
\usepackage[T1]{fontenc}
\usepackage{amsmath,amsthm,amscd,
amssymb,amsfonts,latexsym,mathrsfs}
\usepackage[all]{xy}
\usepackage{color}
\usepackage{ifthen}
\usepackage[normalem]{ulem}
\makeatletter

\@addtoreset{equation}{section}
\def\theequation{\thesection.\@arabic \c@equation}
\def\@citecolor{blue}
\def\@urlcolor{blue}
\def\@linkcolor{blue}
\def\theenumi{\@roman\c@enumi}

\makeatother
\theoremstyle{plain}
\newtheorem{theorem}[equation]{Theorem}
\newtheorem*{definition*}{Definition}
\newtheorem{lemma}[equation]{Lemma}

\newtheorem{proposition}[equation]{Proposition}
\newtheorem*{remark*}{Remark}
\theoremstyle{definition}
\newtheorem{remark}[equation]{Remark}

\newtheorem{remarks}[equation]{Remarks}

\newtheorem{example}[equation]{Example}

\newtheorem{definition}[equation]{Definition}

\def\NZQ{\mathbb}               
\def\NN{{\NZQ N}}

\def\frk{\mathfrak}               
\def\aa{{\frk a}}
\def\bb{{\frk b}}
  
\def\mm{{\frk m}}

\def\opn#1#2{\def#1{\operatorname{#2}}} 
\opn\supp{supp}
\opn\chara{char}
\opn\length{\ell}
\opn\projdim{proj\,dim}
\opn\depth{depth}
\opn\reg{reg}
\opn\lreg{lreg}
\opn\sat{^{sat}}
\opn\lex{^{lex}}
\opn\lra{\longrightarrow}
\opn\thh{^{th}}
\opn\st{^{st}}
\opn\pol{^{\bf p}}
\opn\SK{Sk}
\opn\Char{char}
\opn\Ker{Ker}
\opn\Coker{Coker}
\opn\Im{Im}
\opn\Hom{Hom}
\opn\Tor{Tor}
\opn\Ext{Ext}
\opn\End{End}
\opn\Aut{Aut}
\opn\id{id}
\opn\GL{GL}
\let\:=\colon

\opn\Gin{Gin}
\opn\gin{gin}
\opn\Hilb{Hilb}
\opn\HilbS{HilbS}
\opn\HilbPol{HilbPol}
\opn\ini{in}
\opn\End{end}
\DeclareMathOperator{\hP}{\mathcal H}
\DeclareMathOperator{\iP}{\mathcal I}

\begin{document}

\title{Distractions of Shakin rings}
\author{Giulio Caviglia}
\address{Giu\-lio Ca\-vi\-glia - Department of Mathematics -  Purdue University - 150 N. University Street, West Lafayette - 
  IN 47907-2067 - USA}
\email{gcavigli@math.purdue.edu}
\author{Enrico Sbarra}
\address{Enrico Sbarra - Dipartimento di Matematica - Universit\`a degli Studi di Pisa -Largo Bruno Pontecorvo 5 - 56127 Pisa - Italy}
\email{sbarra@dm.unipi.it}
\thanks{The work of the first author was supported by a grant from the
Simons Foundation (209661 to G. C.)}

\begin{abstract} 
We study, by means of embeddings of Hilbert functions,  a class of rings which we call Shakin rings, i.e. quotients $K[X_1,\ldots,X_n]/\aa$ of a polynomial ring over a field $K$  by ideals $\aa=L+P$  which are the sum of a piecewise lex-segment ideal $L$, as defined by Shakin, and a pure powers ideal $P$. Our main results extend Abedelfatah's recent work on the Eisenbud-Green-Harris conjecture, Shakin's generalization of Macaulay and Bigatti-Hulett-Pardue theorems on Betti numbers and, when $\chara(K)=0$, Mermin-Murai theorem on the Lex-Plus-Power inequality, from monomial regular sequences to a larger class of ideals. 
We also prove an extremality property of embeddings induced by distractions in terms of Hilbert functions of local cohomology modules.
 \end{abstract}
\date{\today}

\maketitle
\

\section*{Introduction}

Hilbert functions are an important object of study in commutative algebra and algebraic geometry  since they encode several fundamental invariants of variates and their coordinate rings such as dimension and multiplicity. A notable result is due to Macaulay \cite{Ma} who provided a characterization of the numerical functions which are Hilbert functions of standard graded algebras, by means of lexicographic (or lex-segment) ideals.
Later, Kruskal and Katona \cite{Kr,Ka} completely characterized the numerical sequences which are $f$-vectors of abstract simplicial complexes, thus 
establishing a remarkable analogue of Macaulay's theorem in algebraic and extremal combinatorics which can be rephrased in terms of Hilbert functions of graded quotients of algebras defined by monomial regular sequences of pure quadrics.

One of the most relevant open problem in the study of Hilbert functions is a conjecture, due to Eisenbud, Green and Harris \cite{EiGrHa1, EiGrHa2}, which aims at extending Kruskal-Katona theorem (and the subsequent generalization of Clements and Lindstr\"om \cite{ClLi}) to a larger class of objects, namely coordinate rings of complete intersections, and obtaining in this way a strong generalization of the Cayley-Bacharach theorem for projective plane cubic curves.
The Eisenbud-Green-Harris conjecture predicts that all Hilbert functions of homogeneous ideals of $R=A/\aa$, where $A$ is a polynomial ring  over a field $K$ and $\aa$ is an ideal of $A$ generated by a homogeneous regular sequence, are equal to those of the images of some lex-segment ideals of $A$ in the quotient ring $A/P$, where $P$ is generated by a certain  regular sequence of pure powers of variables. 

This conjecture, which has been solved in some cases \cite{Ab,CaMa,CaCoVa,Ch,ClLi,FrRi}, renewed a  great deal of interest in understanding and eventually classifying Hilbert functions of quotients  of standard graded algebras $R=A/\aa$, where $A$ is a polynomial ring  over a field $K$ and $\aa$ is a fixed homogeneous ideal of $A,$ in terms of specific properties of $\aa.$ 

In recent years Mermin, Peeva and their collaborators started a systematic investigation of rings $R=A/\aa$ for which all the Hilbert functions of  homogeneous ideals are obtained by Hilbert functions of images in $R$ of lex-segment ideals. 
They called these rings Macaulay-lex \cite{GaHoPe,Me1,Me2,MeMu1,MeMu2,MePe,MePeSt}.  Two typical examples of such rings are the polynomial ring $A$ and the so called Clements-Lindstr\"om rings, i.e. $R=A/P$ where $P=(X_1^{d_1},\dots, X_r^{d_r})$ and $d_1\leq \cdots \leq d_r.$ 

In a polynomial ring $A$, among all the graded ideal with a fixed Hilbert function, the lex-segment ideal enjoys several extremal properties. We summarize some of them in three categories. \\
(1) The lex-segment ideals are the ones with the largest number of minimal generators. Precisely  for a fixed Hilbert function and for every $d$, the lex ideal maximizes the values of $\beta^A_{0d}(-)$, and hence the value of $\beta_0^A(-)$. This fact is a direct consequence of Macaulay's theorem. \\
(2) More generally, by theorems of Bigatti, Hulett \cite{Bi,Hu} (when $\chara(K)=0$) and \cite{Pa}, for every $i,d$ the lex-segment ideal  also maximizes the graded Betti numbers $\beta^A_{id}(-).$\\
(3) Finally, by \cite{Sb1}, the lex-segment ideal   maximizes  the Hilbert functions of the local cohomology modules of $A/(-)$, precisely for every $i$ and $d$ it maximizes  $\dim_K H^i_\mm(A/-)_d$ where $\mm$ is the homogeneous maximal ideal of $A.$ 

When the polynomial ring $A$ is replaced by a Clements-Lindst\"om ring $R=A/P$, we know by \cite{ClLi} that for every Hilbert function, the set of homogeneous ideals with that Hilbert function (if not empty) contains the image, say  $L$, in $R$ of a lex segment ideal of $A.$ The ideal $L$ enjoys extremal properties analogous to the ones discussed above: (1) as a direct consequence of \cite{ClLi},  it  maximizes the values of $\beta^A_{0d}(R/-)$; (2) by \cite{MeMu2}, for all $i$ and $d$, it  maximizes the values of $\beta^A_{id}(R/-)$ and (3) by \cite{CaSb}, for all $i$ and $d$, it maximizes  $\dim_K H^i_\mm(R/-)_d.$

Shakin \cite{Sh}  studied the case of  $R=A/\aa$ where $\aa$ is a piecewise lex-segment ideal, i.e.  the sum over $i$ of the extension to $A$ of lex-segment ideals of $K[X_1,\dots,X_i].$ He showed that such an $R$ is Macaulay-lex, or equivalently that the set of all homogeneous ideals of $R$ with a fixed Hilbert function, when not empty, contains, as in the case of the Clements-Lindstr\"om rings, the image in $R$ of a lex-segment ideal of $A.$ Shakin proved that such an image also satisfies (2), and in particular (1), i.e. it maximizes $\beta^A_{id}(R/-).$ 

In this paper we consider a class of rings which generalizes both the Clements-Lindstr\"om rings and the ones studied by Shakin, namely we study quotients of polynomial rings by the sum of a piecewise lex-segment ideal and a pure power ideal $P=(X_1^{d_1},\dots, X_r^{d_r})$ with $d_1\leq \cdots \leq d_r.$  We call such rings \emph{Shakin rings}. The techniques used in our work are based on the notion of embeddings of Hilbert functions, as defined in \cite{CaKu1}. For instance Macaulay-lex rings are a special example of rings with such embeddings.  

Our first result, Theorem \ref{intermediate},  which we derive as a direct consequence of all the available results on embeddings of Hilbert functions \cite{CaKu1,CaKu2,CaSb}, states that Shakin rings are Macaulay-lex and that they satisfy the properties (1),(2) and (3) mentioned above (with the exception that, to prove (2) when $P\not= (0)$,  we assume $\chara(K)=0$).

The second half this paper is motivated by a recent result of Abedelfatah \cite{Ab}, who proved that the Eisenbud-Green-Harris conjecture holds for distractions, as defined in \cite{BiCoRo}, of Clements-Lindstr\"om ring. 
We prove, in Theorem \ref{maindist},  that the analogous statement (expressed in terms of embeddings of Hilbert functions) holds for Shakin rings. Furthermore we show that distractions of Shakin rings satisfy the analogue of (1), (3),  and under certain assumption (2),  mentioned above.

\section{Embeddings of Hilbert functions and distractions}
Let $A=K[X_1,\dots,X_n]$ be a standard graded polynomial ring over a field $K$ and $\mm=\mm_A$ be its graded maximal 
ideal. Given a homogeneous ideal $\aa\subseteq A$, the quotient ring $A/\aa$ is a standard graded $K$-algebra as well. Aiming at classifying Hilbert functions of standard graded $K$-algebras, we are interested in the study of the poset $\iP_{A/\aa}$ of all homogeneous ideals of $A/\aa$ ordered by inclusion, and of the poset $\hP_{A/\aa}$ of all Hilbert functions of such ideals ordered by the natural point-wise partial order.  In \cite{CaKu1} the problem is approached with the introduction of {\em embeddings (of Hilbert functions)}, which are order-preserving injections $\epsilon\:\hP_{A/\aa} \lra \iP_{A/\aa}$ such that the Hilbert function of $\epsilon(H)$ is equal to $H$, for all $H\in\hP_{A/\aa}$.

The use of embeddings proved to be valuable to extend many significant results known for the polynomial ring to other standard graded $K$-algebras,   \cite{CaKu1,CaKu2,CaSb}. We are therefore interested in understanding for which ideals $\aa$ such an $\epsilon$ exists, and if this is the case we say that {\em the ring $A/\aa$ has an embedding $\epsilon$}. If $I\in \iP_{A/\aa}$ and $\Hilb(I)$ denotes its Hilbert function, with some abuse of notation, we let $\epsilon(I):=\epsilon(\Hilb(I))$. An ideal $I$ is called {\em embedded} when $I\in \Im(\epsilon)$ or, equivalently, $\epsilon (I)=I$. Finally, if $\aa$ is monomial and the pre-image in $A$ of every ideal in $\Im(\epsilon)\subseteq\iP_{A/\aa}$ is a monomial ideal, we say that $\epsilon$ is a {\em monomial embedding}. \\
 
Henceforth $\aa$ will denote a monomial ideal of $A$. 
 
 \begin{remark}\label{embChar} When $A/\aa$ has an embedding $\epsilon$, then $A/\aa$ also has a monomial embedding obtained by composing $\epsilon$ with the operation of taking the initial ideal with respect to any fixed monomial order.
Let $K$ and $\tilde K$ be two fields and consider two monomial ideals $\aa \subseteq A=K[X_1,\dots,X_n]$ and $\bb \subseteq B=\tilde K[X_1,\dots,X_n]$ generated by the same set of monomials.  Then, $A/\aa$ has an embedding if an only if $B/\bb$ does; in fact, any monomial embedding of $A/\aa$ induces (the same) monomial embedding in $B/\bb$ and vice-versa, since Hilbert functions of monomial ideals do not depend on the ground field. 
\end{remark}

\begin{remark}\label{embquoz}
Let $R$ be a ring with an embedding $\epsilon,$ and let $I\in\Im(\epsilon)$. 
The ring $R/I$ has a natural embedding $\epsilon'$ induced by $\epsilon$: if $\pi\: R\lra R/I$ denotes the canonical projection, one defines 
$\epsilon'(J):=\pi(\epsilon(\pi^{-1}(J))$, for all $J\in\iP_{R/I}$. By applying $\epsilon$ to $I\subseteq \pi^{-1}(J)$ one gets $I\subseteq \epsilon(\pi^{-1}(J))$ and, thus, $\Hilb(J)=\Hilb(\epsilon'(J))$. Moreover, if $J, H\in\iP_{R/I}$ with $\Hilb(J) \leq \Hilb(H)$, then $\Hilb(\pi^{-1}(J))\leq \Hilb(\pi^{-1}(H))$; hence $\epsilon(\pi^{-1}(J))\subseteq \epsilon(\pi^{-1}(H))$ and, consequently, $\epsilon'(J)\subseteq \epsilon'(H)$; therefore $\epsilon'$ is an embedding, since it preserves Hilbert functions and inclusions of ideals.
\end{remark}

We will study distractions of monomial ideals, as introduced in \cite[Def. 2.1]{BiCoRo}.

A {\em distraction $D$ (of $A$)} is an $n\times \NN^*$ infinite matrix whose entries $l_{ij}\in A_1$ verify $(i)$ for all choices of $j_1,\ldots,j_n\in \NN^*$,  $l_{1j_1},\ldots,l_{nj_n}$ form a system of generators of the $K$-vector space $A_1$; (ii) there exists $N\in\NN^*$ such that, for every $i=1,\ldots,n$, the entries $l_{ij}$ are constant for all $j\geq N$. Given a monomial ${\bf X}^{\bf a}=X_1^{a_1}X_2^{a_2}\cdots X_n^{a_n}\in A$ we define its distraction to be the polynomial 
$D({\bf X}^{\bf a})=\prod_{i=1}^n\prod_{j=1}^{a_i}l_{ij}$ and we extend $D$ by $A$-linearity to a map from $A$ to $A$.
By \cite[Cor. 2.10]{BiCoRo}, when $I$ is a monomial ideal, the {\em distraction $D(I)$} 
is the homogeneous ideal  generated by the distractions of a monomial system of generators of $I$; furthermore  $\Hilb(D(I))=\Hilb(I).$
It is immediate to see that $D$ preserves inclusions of ideals.

\begin{remark}\label{udsrd}
Let  $R=A/\aa$ be  a ring with an embedding $\epsilon$; also, let $\bb$ be another homogeneous ideal of $A$ and $S=A/\bb$; we are interested in studying relations between $\hP_R$ and $\hP_S$, and between $\iP_R$ and $\iP_S$, when $\bb=D(\aa)$ for a distraction $D$ of $A.$ By fixing any monomial order $\prec$ on $A$ and applying $D$ to the initial ideal of the pre-image in $A$ of an ideal $I\in \iP_R$, we immediately see that  $\hP_R\subseteq\hP_S$.  When  $\hP_S\subseteq \hP_R$ (and, therefore, $\hP_S = \hP_R$) then $S$ has an embedding, say $\epsilon_D$, induced by $\epsilon$ and $D$:
\[
\begin{CD} \hP_R @>\epsilon>> \iP_R @>\pi^{-1}>> \iP_A\\
@|  @. @VV D\; \circ\; \ini_\prec  V \\
\hP_S @>\epsilon_D>> \iP_S @<\pi<<\iP_A
\end{CD}
\]
\end{remark}

From now on, given a Hilbert function $H$, we will denote with $H_d$ its value at $d$, so that, when $M$ is a graded module, $\Hilb(M)_d=\dim_K(M_d).$ Furthermore, the Hilbert series of $M$ will be denoted by  
$\HilbS(M)$, i.e. $\HilbS(M)=\sum_{d\in \mathbb Z} \Hilb(M)_dz^d.$

\begin{remark}\label{bedistra}
Recall that, given a finitely generated graded $A$-module $M$, the {\em $(i,j)\thh$ graded Betti number
$\beta_{ij}^A(M)$} of $M$ is defined as $\Hilb\left(\Tor^A_i(M,K)\right)_j$.  
By \cite[Cor. 2.20]{BiCoRo}, for all distractions $D$ of $A$ and for all $i,j$, one has  $\beta^A_{ij}(A/\aa)=\beta^A_{ij}(A/D(\aa)).$ 
\end{remark}

It is important to observe that the distraction of a monomial ideal can be obtained, as described below, as a polarization (see \cite[sect. 1.6]{HeHi}) followed by a specialization.  

We let the polarization of a monomial ideal $\aa\subseteq A$, denoted by $P(\aa)$, be the ideal of $T=A[X_{1 1},\dots,X_{1 r_1},\dots,X_{n\,1},\dots  X_{n r_n}]$ generated by the monomials $\prod _{i=1}^n(\prod_{j=1}^{a_i} X_{i j})$ for which $\prod _{i=1}^n X_i^{a_i}$ is a minimal generator of $I;$ we have chosen $r_i$ to be equal to $0$, if no minimal monomial generator of $\aa$ is divisible by $X_i,$ or otherwise the maximum exponent $a>0$ such that $X_i^a$ divides a minimal monomial generator of $\aa.$
The elements of the set $\mathcal X= \{ X_{i}-X_{i j}~:1\leq i\leq n, ~j\geq 1\}$  form a regular sequence for $T/P(\aa)$, moreover, since there is a graded isomorphism $A/\aa \simeq   T/(P(\aa)+(\mathcal X ))$,  we have that  $\HilbS(T/P(\aa))=\HilbS(R/\aa)/(1-z)^r$ where $r=\vert \mathcal X \vert =(\sum_{i=1}^n r_i).$ 

Now consider a distraction matrix $D$ for $A$ with entries $l_{i j}$ 
and notice that $D(\aa)$ is generated by the forms $\prod _{i=1}^n(\prod_{j=1}^{a_i} l_{i j})$ for which $\prod _{i=1}^n X_i^{a_i}$ is a minimal monomial generator of $I.$ 
We have already mentioned that $\Hilb(A/\aa)=\Hilb(A/D(\aa)),$ hence we can deduce that the $r$ linear forms of the set $\mathcal L= \{l_{ij}-X_{ij}:1\leq i\leq n, ~j\geq 1\}$ are a regular sequence for 
$T/P(\aa)$ because we have a graded isomorphism $A/D(\aa)\simeq T/(P(\aa)+ (\mathcal L))$ and $A/D(\aa)$ has the expected Hilbert series $\Hilb(T/P(\aa))(1-z)^r$.

We are interested in comparing, for all $i$ and for all distractions $D$, the Hilbert functions  of the local cohomology modules $H^i_{\mm_A}(A/\aa)$ and  $H^i_{\mm_A}(A/D(\aa)).$ 

\begin{proposition}\label{codistra}
Let $\aa$ be a monomial ideal of $A$. Then, for all distractions $D$, one has $$\Hilb \left(H^i_{\mm_A}(A/\aa)\right)_j\leq \Hilb \left(H^i_{\mm_A}(A/D(\aa))\right)_j, \hbox{\;\;\; for all \;}i, j.$$
\end{proposition}
\begin{proof} 
We adopt the same notation as the above discussion.
We can  extend the field, without changing the Hilbert functions under consideration, and assume  $\vert K \vert =\infty.$ By  \cite[Cor. 5.2]{Sb1} we know that  $\HilbS\left(H^i_{\mm_A}(A/\aa)\right)=(z-1)^r\HilbS\left(H^{i+r}_{\mm_T}(T/P(\aa))\right).$
Let  $g$ be the change of coordinates of $T$ which is the identity on $A$ and sends, for every $i$ and $j$,  $X_{ij}$ to  $X_{ij}+ l_{ij}.$ Let $\mathbf{w} =(w_1,\dots,w_{n+r})$ be a weight such that $w_i=1$ when $i\leq n$ and $w_i=0$ otherwise. Let $\bb$ be the ideal $D(\aa)T.$ Notice that $\bb\subseteq in_{\mathbf w}(g(P(\aa)))$ and since these two ideals have both Hilbert series equal to $\HilbS(A/\aa)/(1-z)^r$, they are equal a well. By \cite[Thm. 2.4]{Sb1} we obtain:
$\Hilb\left(H^{i+r}_{\mm_T}(T/P(\aa))\right) \leq  \Hilb\left(H^{i+r}_{\mm_T}(T/\bb) \right)$.
 By \cite[Lemma 2.2]{Sb2} we have  $\HilbS\left(H^{i+r}_{\mm_T}(T/\bb) \right)= (\sum_{h<0}z^h)^r \HilbS\left(H^{i}_{\mm_{A}}(A/ D(\aa)) \right).$
 Finally since $(\sum_{h<0}z^h)^r(z-1)^r=1$ we obtain the desired inequality.
\end{proof}

\section {Embeddings and ring extensions} 

We start this section by recalling some definitions about embeddings of Hilbert functions, which were introduced in \cite{CaKu1,CaKu2} and \cite{CaSb}. 

Let $R=A/\aa$, where $\aa$ is not necessarily a monomial ideal. 

If $\aa$ is the $0$ ideal then Macaulay's theorem implies that the ring $R=A$ has the monomial embedding $\epsilon$, which maps an Hilbert series $H$ to the unique lexicographic-segment ideal of $\iP_A$ with Hilbert series $H$. This fact motivates the following definition. 
Let $\aa$ be a monomial ideal, $\pi$ the canonical projection of $A$ onto $R$, and assume that $R$ has an embedding $\epsilon$. Then, $\epsilon$ is called {\em the lex-embedding} if 
\[\Im(\epsilon)=\left\{ \pi(L)\in\iP_R \: L\in\iP_A,\;L \hbox{ lex-segment ideal}\right\}.\]
In other words, $R$ has the lex-embedding precisely when $R$ is Macaulay-lex in the sense of \cite{MePe}. 

Assume that $R$ is a ring with an embedding $\epsilon$ and let $S=A/\bb$ be another standard $K$-algebra with $\hP_{S}\subseteq \hP_{R}$; we write  $(S,R,\epsilon)$ and observe that, via $\epsilon$, we may associate to an ideal of $S$ an ideal of $R.$
The following definitions were introduced in \cite{CaSb}. We say that $(S,R,\epsilon)$ (or simply $\epsilon$) is {\em (local) cohomology extremal} if,  for every homogeneous ideal $I$ of $S$, one has $\Hilb \left(H^i_{\mm_A} (S/I)\right)_j\leq \Hilb \left(H^i_{\mm_A}(R/\epsilon(I))\right)_j$, for all $i, j$.  It is easy to see that, if $R$ is Artinian so is $S.$ Furthermore $(S,R,\epsilon)$ and, thus, $(R,\epsilon)=(R,R,\epsilon)$ is cohomology extremal: in this case the only non-zero local cohomology module of $R/\epsilon(I)$ is $H^0_{\mm_A}(R/\epsilon(I))= R/\epsilon(I)$, for all ideals $I$. 
 
Similarly,  if for all homogeneous ideals $I\subseteq S$ and for all $i,j$ one has  $\beta^A_{ij}(S/I)\leq \beta^A_{ij}(R/\epsilon(I))$, then $(S,R,\epsilon)$,  and $(R,\epsilon)$ when $S=R$, is said to be {\em Betti extremal}. 

\begin{remark}\label{embeddeddoso}
Let $(R,\epsilon)$ be a ring with an embedding,  $I\in  \Im(\epsilon)$ and $\epsilon'$  as in Remark \ref{embquoz}. Since $I$ is embedded and $R/\pi^{-1}(J)\simeq (R/I)/J$ for all $J\in\iP_{R/I}$, it is easy to see that, if $(R,\epsilon)$ is Betti or cohomology extremal, then $(R/I,\epsilon')$ is as well Betti or cohomology extremal.

Similarly if  $(S,R,\epsilon)$ is Betti or cohomology extremal, and $H\subseteq S$ is a homogeneous ideal, then  $(S/H, R/\epsilon(H), \epsilon')$ is as well Betti or cohomology extremal.
\end{remark}

We can now summarize, and we do in Theorem \ref{Exte}, some known results about embeddings and we refer the reader to \cite{CaKu1}, \cite{CaKu2} and \cite{CaSb} for a general treatise. We start by recalling a crucial definition for what follows.  

Let $\bar A=K[X_1,\dots,X_{n-1}]$ and let $\bar R$ be $\bar A/ \bar \aa$ for a homogeneous ideal $\bar \aa \subseteq \bar A.$ Given $e\in \NN\cup\{\infty\}$ we let  $R=\bar R [X_n]/(X_n^e)$, where $(X_{n}^e)$ denotes the zero ideal when $e=\infty$. A homogeneous ideal $J$ of $R$ is called {\em $X_n$-stable} if it can be written as 
$\bigoplus_{d\in\NN} J_{[d]}X_n^d$ where each $J_{[d]}$ is an ideal of $\bar R$ and for all $0< k+1<e$ the inclusion  $J_{[k+1]}\mm_R\subseteq  J_{[k]}$ holds, cf. \cite[Def. 3.2]{CaKu1}, \cite[Def. 1.1]{CaSb}.

\begin{remark} \label{Stab} Let $e=\infty$. 
{\rm (i)} By \cite[Lemma 4.1]{CaKu1} for every homogeneous ideal $I$ of $R$ there exists a $X_n$-stable ideal $J$ of $R$ with the same Hilbert function as $I$. {\rm (ii)} The discussion after \cite[Thm. 3.1] {CaKu2} yields that  $\beta_{ij}^A(R/I)\leq \beta_{ij}^A(R/J)$, for all $i, j$. {\rm (iii)} By \cite[Prop. 1.7]{CaSb}, one also has  $\Hilb\left(H^i_{\mm_A}(R/I)\right)\leq \Hilb\left(H^i_{\mm_A}(R/J)\right)$ for all  $i$. 
\end{remark}

\begin{theorem}\label{Exte}
Let $\bar R= \bar A/ \bar \aa$ be a ring with an embedding $\bar \epsilon$ and $R=\bar R[X_n]$. Then, $R$ has an embedding $\epsilon$ such that, for every homogeneous ideal $I$ of $R$, we have $\epsilon(I)=\bigoplus_{d\geq 0} \bar J_{[d]}X_n^d$, where each $\bar J_{[d]}\in \Im(\bar \epsilon).$
Furthermore,
{\rm (1)} if $\bar \epsilon$ is the lex-embedding, then $\epsilon$ is the lex-embedding; {\rm (2)} if $\bar \epsilon$ is cohomology extremal, then $\epsilon$ is cohomology extremal; {\rm (3)} if $\bar \epsilon$ is Betti extremal, then $\epsilon$ is Betti extremal.
\end{theorem}
\begin{proof} The existence of $\epsilon$ follows from \cite[Thm. 3.3]{CaKu1} together with Remark \ref{Stab} and \cite[Remark 2.3]{CaKu2}. With the assumption in {\rm (1)} it has been proven in \cite[Thm 4.1]{MePe} that $R$ has the lex-embedding; to see that the above $\epsilon$ coincides with the lex-embedding of $R$, we  notice that $\bar \epsilon$ induces an embedding order on $\bar R$ (see \cite[Discussion 2.15]{CaKu1}), which is a monomial order  in the sense of \cite[Def. 1.2]{CaKu1}; by \cite[Thm. 3.11]{CaKu1}, $\epsilon$ induces a monomial order on $R$ as well and, finally, by \cite[Prop. 2.16]{CaKu1}, $\epsilon$ is the lex-embedding.
Part {\rm (2)} is a special case of \cite[Thm. 3.1]{CaSb}, namely when there is only one ring. Finally, by Remark \ref{Stab}{\rm (ii)}, {\rm (3)} is a consequence of \cite[Thm. 3.1]{CaKu2}.
\end{proof} 

\section{Embeddings of Shakin rings}

From now on we let $\aa$ be a monomial ideal of $A=K[X_1,\dots,X_n]$ and  $R=A/\aa$.

\begin{definition}\cite[Def. 2.1 and Prop. 2.4]{Sh}
For $i=1,\dots,n$, let $A_{(i)}= K[X_1,\ldots,X_i]\subseteq A$. An ideal of $A$ is called {\it piecewise lex-segment} (or {\em piecewise-lex} for short) if it can be written as a sum of (possibly zero) monomial ideals  $L_1,\ldots,L_n$, where  for every $i$, $L_i=L_{(i)}A$ and $L_{(i)}$ is a lex-segment ideal of $A_{(i)}$. 
\end{definition}

It is proven in \cite[Thm 3.10]{Sh}  that, if $\aa$ is a piecewise-lex ideal, then Macaulay's Theorem holds for $R=A/\aa$. Moreover, in \cite[Thm 4.1]{Sh}, it is proven that Bigatti-Hulett-Pardue result on extremality of Betti numbers of lex-segment ideals of $A$ extends to $R=A/\aa$, whenever $\aa$ is a piecewise-lex ideal and $\chara(K)=0.$ 
By using embeddings, it is possible to prove these results for a larger class of ideals, which we introduce in the next definition.

\begin{definition}\label{shakinideal} 
We call an ideal $\aa \subseteq A$  a {\it Shakin ideal} if there exist a piecewise-lex ideal $L$ and a pure powers ideal $P=(X_1^{d_1},\ldots, X_r^{d_r})$, $d_1\leq d_2\leq \cdots \leq d_r$, of $A$ such that $\aa=L+P$. If this is the case, we call the quotient ring $A/\aa$ a {\it Shakin ring}. 
\end{definition}

\noindent
The following remark is an analogue of Remark \ref{Stab}{\rm (i,ii)}.

\begin{remark}\label{StabP}
Let  $\chara(K)=0$ and let $R= A/\aa$ be a Shakin ring such that $r=n$ and $L_{(n)}=0$. Then we may write $R$ as $\bar R[X_n]/(X_n^{d_n})$, where $\bar R=\bar A /\bar \aa$ is a Shakin ring. By \cite{CaKu1}, proofs of Lemmata 4.1 and 4.2 and by the discussion after Theorem 3.1 in \cite{CaKu2}, we know that, for every homogeneous ideal $I$ of $R$, {\rm (i)} there exists an $X_n$-stable ideal $J$ of $R$  such that $\Hilb(I)=\Hilb(J)$, and {\rm (ii)}  $\beta^A_{ij}(R/I)\leq \beta^A_{ij}(R/J)$ for all $i, j$. {\rm (iii)} In this setting, if $\bar R$ has the lex-embedding, so does $R$ and the proof runs as that of Theorem \ref{Exte} (1): note that $\bar \epsilon$ induces an embedding order on $\bar R$ (\cite[Discussion 2.15]{CaKu1}), which is  a monomial order  in the sense of \cite[Def. 1.2]{CaKu1}; finally by \cite[Thm. 3.11]{CaKu1}, $\epsilon$ induces a monomial order on $R$ as well, which by \cite[Prop. 2.16]{CaKu1} implies that  $\epsilon$ is the lex-embedding.
\end{remark}

\begin{theorem}\label{intermediate}
Let $R=A/\aa$ be a Shakin ring. Then, {\rm (1)} $R$ has the lex-embedding; {\rm (2)} such an embedding is cohomology extremal; {\rm (3)} if $\chara (K)=0$ or $P=0$, then such an embedding is also Betti extremal.
\end{theorem} 

\begin{proof} We use induction on the number $n$ of indeterminates. If $n=0$ there is nothing to prove, and if $n=1$   the results are trivial since one can only set $\epsilon(I)=I$ for all $I\in\iP_R$. Let us now assume $n>1$ and $\aa=L+(X_1^{d_1},\dots,X_r^{d_r}),$ where $L$ is a piecewise-lex ideal; we may write $\aa=\bar \aa A + L_{(n)} + Q$ where $\bar \aa$ is a Shakin ideal of $\bar A$,  $L_{(n)}$ is a lex-segment ideal of $A$, whereas  $Q=(0)$ if $r<n$ and $Q=(X_n^{d_n})$ otherwise. By the induction hypothesis, {\rm (1)}, {\rm (2)} and {\rm (3)} hold for $\bar R=\bar A / \bar \aa$, and also for the ring $\bar R[X_n]\simeq A/\bar \aa A$ by Theorem \ref{Exte}. In particular, $A/\bar \aa A$ has the lex-embedding. 

Next, we are going to show that the three claims also hold, when going modulo $Q\neq 0$, for the ring 
$S=A/(\bar\aa A +Q)\simeq \bar R[X_n]/(X_n^{d_n})$. 
In order to prove {\rm (1)}, it is not restrictive to assume $\chara(K)=0$, see Remark \ref{embChar} and {\rm (1)} holds for $S$ by Remark \ref{StabP}(iii).  For the second claim, we only need to say that $S$ is Artinian and, thus, any embedding is cohomology extremal. By Remark \ref{StabP}{\rm (ii)}, \cite[Thm. 3.1]{CaKu2} yields that claim {\rm (3)} holds for $S$.

Finally, since  $R\simeq S/L_{(n)}S$, it is sufficient to observe, as we did in Remark \ref{embeddeddoso}, that the three claims behave well when modding out by an embedded ideal, and $L_{(n)}$ is such, since $S$ has the lex-embedding.
\end{proof}
The previous result, part {\rm (1)} and {\rm (3)}, extends \cite[Thm 3.10, Thm 4.1]{Sh} from piecewise-lex ideals to Shakin ideals. Part {\rm (3)} also  extends \cite[Thm 3.1]{MeMu2}, from pure powers ideals to Shakin ideals. 

We believe that the conclusion of Theorem \ref{intermediate} {\rm (3)} should also hold in positive characteristic.

\section{Distractions and Shakin rings}
A recent result of Abedelfatah on the Eisenbud-Green-Harris conjecture, \cite[Cor. 4.3]{Ab}, can be rephrased as follows: when $\aa\subseteq A$ is a pure powers ideal $(X_1^{d_1},\dots,X_n^{d_n})$ where $d_1\leq \cdots \leq d_n$, then, for every distraction  $D$,  one has $\hP_{A/D(\aa)}\subseteq \hP_{A/\aa}$. By Clements-Lindstr\"om theorem $A/\aa$ has the lex-embedding and therefore, by Remark \ref{udsrd}, $\hP_{A/D(\aa)}=\hP_{A/\aa}$ and $A/D(\aa)$ has an embedding. 
We shall show in Theorem \ref{maindist} (1) that the same result is valid, more generally, for any Shakin ring. 
\vspace{.2cm}

The following result is a simple fact that will be crucial for the proof of Theorem \ref{abed}.

\begin{lemma}[Gluing Hilbert functions]\label{glue} 
 Let $R$ be a ring with an embedding $\epsilon$. Let $\{_dI\}_{d\in \mathbb N}$ be a collection of homogeneous ideals of $R$ with the property that, for all $d$, the Hilbert function of $_dI$ and of $_{d+1}I$ are equal in degree $d+1$. Then, there exists an ideal $L$ such that, for every $d$, the Hilbert functions of $L$ and that of $_dI$ are equal in degree $d$. 
 \end{lemma}
 \begin{proof}
 By \cite[Lemma 2.1]{CaKu1}, the ideals  $\epsilon(_dI)$ and  $\epsilon(_{d+1}I)$ coincide in degree $d+1$. Thus, $ R_1\epsilon(_dI)_d  \subseteq \epsilon(_dI)_{d+1}=\epsilon(_{d+1}I)_{d+1}$ and the direct sum $\bigoplus_{d\in \mathbb N} \epsilon(_dI)_d$ of the vector spaces $\epsilon(_dI)_d$, $d\in\NN$,  is the ideal $L$ of $R$ we were looking for.
 \end{proof}

As before, we let $e\in \NN\cup\{\infty\}$ and, when $e=\infty$, we let the ideal $(X_{n}^e)$ denote the zero ideal. 

\begin{theorem}\label{abed} Let $\bar \aa \subseteq \bar A=K[X_1,\dots,X_{n-1}]$ be a monomial ideal, $\bar R= \bar A/\bar \aa$. Let also $A=\bar A[X_{n}]$,  $\aa\subseteq A$ be the ideal $\bar \aa A+X_{n}^eA$, and $R=A/\aa$. Suppose that  $\bar R$ and $R$ have  embeddings.
If $\hP_{\bar A/\bar D( \bar \aa)}= \hP_{\bar A/\bar \aa}$ for every distraction  $\bar D$ of $\bar A$, then  $\hP_{A/D(\aa)}= \hP_{A/\aa}$ for every distraction  $D$ of $A$. 
\end{theorem}
\begin{proof} Since $\bar \aa$ is a monomial ideal, by Remark \ref{embChar} we may assume that $\bar R$ has a  monomial embedding $\bar \epsilon$.
 Let us denote by $\epsilon$ the embedding of $R$.  By Remark \ref{udsrd}, it is enough to show  that, for every distraction   $D$, one has $\hP_{A/D(\aa)} \subseteq \hP_{A/(\aa)}$.  Let $I$ be an ideal of   ${A/D(\aa)}$ and let $J$ be its pre-image in $A$, we are going to show that  there exists an ideal $L$ of $A$, which contains $\aa$ and with the same Hilbert function as $J$.

\noindent
\framebox{$e=\infty$} Let ${\bf \omega}$ be the weight vector $(1,1,\ldots,1,0)$ and fix a change of coordinates $g$ such that  $gD(X_{n})=X_{n}$. 
Now, we decompose the ideal $\ini_{\bf \omega}(gJ)$ of $A$ as the (not finitely generated) $\bar A$-module $\ini_{\bf \omega}(gJ)=\bar J_{[0]} \oplus \bar J_{[1]}X_{n} \oplus \cdots \oplus \bar J_{[i]}X_{n}^i\oplus\cdots.$ 

It is a standard observation that, for all $i$, the ideal $\bar J_{[i]}$ is the image in $\bar A$ of the homogeneous ideal $ {(gJ):X_{n}^i}$ under the map evaluating $X_n$ at $0.$ In particular $\bar J_{[i]}\subseteq \bar J_{[i+1]}$, for all $i$. 

\begin{remark}\label{udsre}
Notice that $X_n$ is an entry of the last row of the distraction $gD.$ 
Thus, one can easily verify that, if we map all the entries of $gD$ to $\bar A$ evaluating $X_n$ at zero, we get a matrix whose first $n-1$ rows form a distraction of $\bar A$; we denote it by $\bar D$. Since $\bar D (\bar \aa)$ is the image in $\bar A$ of $gD(\aa)$, and $D(\aa)\subseteq J$, we have that $\bar D(\bar \aa)\subseteq \bar J_{[0]}\subseteq \cdots\subseteq \bar J_{[i]}\subseteq \cdots$.
\end{remark}

We can now continue with the proof of the theorem.
The above chain determines a chain of ideals in $\bar A/\bar D(\bar \aa)$, and thus a chain of elements in  $\hP_{\bar A/\bar D(\bar \aa)}=\hP_{\bar A/(\bar \aa)}$. By applying the embedding $\bar \epsilon$ of $\bar A/\bar \aa$ to the latter and lifting the resulting chain in $\iP_{\bar A/\bar\aa}$ to a chain in $\iP_{\bar A}$, we get $\bar \aa \subseteq \bar L_{[0]}\subseteq\cdots\subset \bar L_{[i]} \subseteq \cdots$, where $\bar L_{[i]}$ is a monomial ideal of $\bar A$ for all $i$.  Now, the ideal $L:=\bar L_{[0]} \oplus \bar L_{[1]}X_{n} \oplus \cdots$ contains $\aa$ and has the desired Hilbert function.

\noindent
\framebox{$e\in \mathbb N$} We may let $e\geq 1$, since the conclusion is trivial for $e=0$. 
By Lemma \ref{glue}, is enough to show that,  
for every positive integer $d$, there exists an ideal $L\supseteq \aa$ of $A$ whose Hilbert function agrees with the one of $J$ in degrees $d$ and $d+1$.

Let $d$ be a fixed positive integer and let $D(X_{n}^e)=l_1\cdots l_e$,  where $l_i\in A_1$ for $i=1,\ldots,e$.  By re-arranging these linear forms, if necessary, we may assume without loss of generality that
$$\dim_K(J+(l_1))_{d+1} \leq \dim_K(J+(l_j))_{d+1} \quad  \text{ for all }j=2,\dots,e,$$
and, recursively, that, for $h=1,\ldots,e$, 
\begin{equation*}\label{EQ1}
\dim_K(J: \prod_{i=1}^{h-1} l_i + (l_h))_{d-h+2}  \leq 
\dim_K(J: \prod_{i=1}^{h-1} l_i + (l_j))_{d-h+2},  
\end{equation*} for all $j=h+1,\dots,e$.
The latter equation implies that 

\begin{equation}\label{EQ2} 
\dim_K(J: \prod_{i=1}^{h-1} l_i + (l_h))_{d-h+2}  \leq 
\dim_K(J: \prod_{i=1}^{h-1} l_i + (l_{h+1}))_{d-h+2} \leq  \dim_K(J: \prod_{i=1}^{h} l_i + (l_{h+1}))_{d-h+2} 
\end{equation} for $h=1,\dots,e-1$. 
Furthermore, for $h=1,\ldots,e$, we have short exact sequences
$$0\lra \left(A/(J:\prod_{i=1}^hl_i)\right)(-1) \stackrel{\cdot l_h}{\lra} A/ (J:\prod_{i=1}^{h-1}l_i) \lra A/ (J:\prod_{i=1}^{h-1}l_i+(l_h)) \lra 0.$$ Notice that $(J:\prod_{i=1}^{e}l_i)=A.$
The additivity of Hilbert function for short exact sequences, thus, implies that the Hilbert function of $J$ can be computed  by means of those of $J:\prod_{i=1}^{h-1}l_i+(l_h)$, $h=1,\ldots, e$, and that of $A$.

For every $h=1,\dots,e$, we let $g_h$ be a change of coordinates of $A$ such that $g_h(l_{h})=X_{n}$, and we denote by $\bar J_{[h-1]}$ the image of $g_h (J: \prod_{i=1}^{h-1} l_i+(l_h))$ in $\bar A$; for all $h\geq e$, we also set $\bar J_{[h]}=\bar A$. 
With these assignments, one verifies that the Hilbert function of $J$ is the same as the Hilbert function of the $\bar A$-module $\bar J_{[0]} \oplus \bar J_{[1]}X_{n} \oplus \cdots$; the difference with the case $e=\infty$ is that we cannot conclude that $\bar J_{[i]}\subseteq \bar J_{[i+1]}$ for all $i\geq 0$, but \eqref{EQ2} yields that
$ \dim_K (\bar J_{[h]})_{d-h+1}\leq \dim_K  (\bar J_{[h+1]})_{d-h+1}$ for all $h=0,\dots, e-1$. Since the inequality is also true for $h\geq e$, we may conclude that
\begin{equation*}
 \dim_K (\bar J_{[h]})_{d-h+1}\leq \dim_K  (\bar J_{[h+1]})_{d-h+1} \quad \quad \text{  for all } h\geq 0.
\end{equation*}
 
Furthermore, by Remark \ref{udsre} applied to $g_hD$, for all $h\geq 1$,  there exists a distraction $\bar D_h$ of $\bar A$ such that $\bar D_h (\bar \aa) \subseteq \bar J_{[h-1]}$. 
Since $\hP_{\bar A/\bar D_h(\bar \aa)}= \hP_{\bar A/\bar \aa}$ by hypothesis, and $\bar A/\bar \aa$ has an embedding  $\bar \epsilon$, we can let, for all $h\geq 1$,  $\bar L_{[h-1]}$ be the pre-image in $\bar A$ of $\bar \epsilon(\bar J_{[h-1]})$; therefore $\bar L_{[h]}\supseteq \bar\aa$ and  $\dim_K (\bar L_{[h]})_{d-h+1}\leq \dim_K (\bar L_{[h+1]})_{d-h+1}$ for all $h\geq 0$. By \cite[Lemma 2.1]{CaKu1}, any homogeneous component of an embedded ideal is uniquely determined by the value of the given Hilbert function in that degree, hence 
\begin{equation}\label{EQ4}
(\bar L_{[h]})_{d-h+1}\subseteq (\bar L_{[h+1]})_{d-h+1} \quad \quad \text{  for all } h\geq 0.
\end{equation}

We thus can define the $\bar A$-module $N=\bar L_{[0]} \oplus \bar L_{[1]}X_{n} \oplus \cdots \subseteq A$, and, by \eqref{EQ4},  $\mm_A N_d =(\mm_{\bar A}+X_n)N_d \subseteq N_{d+1}$. Furthermore since $\aa=\bar \aa A+X_n^eA$, $\bar \aa \subseteq \bar L_{[h]}$ for all $h\geq 0$  and $\bar L_{[h]}=\bar A$ for all $h\geq e$, we have that $\aa \subseteq N.$ We let $L$ be the ideal of $A$ generated by $N_d$, $N_{d+1}$ and $\aa$ and we notice that  $L_d=N_d$ and $L_{d+1}=N_{d+1}.$ Finally $L\supseteq \aa$ is the desired ideal because its Hilbert function agrees with that of $J$ in degree $d$ and $d+1.$
\end{proof}

Let as before $\bar A=K[X_1,\ldots,X_{n-1}]$, and let $\bar S=\bar A/\bar \bb$ and $\bar R=\bar A/\bar \aa$ be standard graded algebras such that $\bar R$ has an embedding $\bar \epsilon$ and $\hP_{\bar S}\subseteq \hP_{\bar R}$, so that, as in Section 2, we can consider the triplet $(\bar S,\bar R,\bar\epsilon)$. In the proof of the following theorem we shall need a technical result about extension of embeddings we proved in \cite[Thm. 3.1]{CaSb}: if  $(\bar S, \bar R, \bar\epsilon)$ is cohomology extremal, then $(\bar S[X_n], \bar R[X_n],\epsilon)$ is cohomology extremal,  where $\epsilon$ is the usual extension of the embedding $\bar \epsilon$ to $\bar R[X_n]$ which has been consider throughout this paper.

We are now ready to present our main theorem about distractions of Shakin rings; its proof follows the outline of that of Theorem \ref{intermediate} and makes use of Theorem \ref{abed}. The reader should keep in mind the construction of an embedding induced by a distraction we presented in Remark \ref{udsrd}.

\begin{theorem}\label{maindist} Let $A/\aa$ be a Shakin ring and $\aa=L+P$. Then, for every distraction $D$ of $A$, one has: {\rm (1)} $\hP_{A/D(\aa)}= \hP_{A/\aa}$ and $A/D(\aa)$ has an embedding $\epsilon_D$; {\rm (2)} the embeddings $(A/D(\aa),\epsilon_D)$ and  $(A/D(\aa),A/\aa,\epsilon)$ are cohomology extremal; {\rm (3)} if $P=0$, i.e. $\aa=L$ is a piecewise-lex ideal, then $(A/D(\aa),\epsilon_D)$  and  $(A/D(\aa),A/\aa,\epsilon)$ are Betti extremal.
\end{theorem}

\begin{proof} For clarity's sake, we split the proof in several steps.

(a) We induct on $n$. If $n=0, 1$ the results are trivial. Let $P=(X_1^{d_1},\dots,X_r^{d_r})$ and $d_1\leq\cdots\leq d_r.$ We may write $\aa$ as $\bar \aa A + L_{(n)} + Q$, where $\bar \aa$ is a Shakin ideal of $\bar A$,  $L_{(n)}$ is a lex-segment ideal of $A$ and $Q=(0)$ if $r<n$ or $Q=X_n^{d_n}A$ otherwise. Notice that the Shakin ring $\bar R=\bar A / \bar \aa$ has the lex-embedding, say $\bar\epsilon$, by Theorem \ref{intermediate}; and, by the inductive hypothesis, it  satisfies {\rm (1)}, {\rm (2)}, and {\rm (3)}. 

(b) We let now $R$ be the Shakin ring $R=A/(\bar \aa A +Q)$. By Theorem \ref{intermediate}, $R$ has the lex-embedding, say $\epsilon$ and, therefore,  $L_{(n)}R$ is embedded. For any distraction $D$ of $A$, let $R_D:=A/D(\bar \aa A +Q)$: 
if  $\hP_{R_D}=\hP_R$, then the image of $D(L_{(n)})$ in $R_D$ is also embedded via $\epsilon_D$, i.e. it belongs in $\Im(\epsilon_D)$ - cf. Remark \ref{udsrd}. Thus, by virtue of Remark  \ref{embeddeddoso}, in order to conclude the theorem, it is enough to prove (1), (2), and (3) for $R$. Without loss of generality we assume $\aa=\bar \aa A +Q.$

(c) By step (a), we may apply Theorem \ref{abed} to $\bar R$ and $R=\bar R [X_n]/(Q)$, hence Remark \ref{udsrd}  yields that {\rm (1)} holds for $R$. If $Q\not =0$, then $R$ is Artinian and so is $R_D$; therefore,  $(R,\epsilon)$, 
$(R_D,\epsilon_D)$ and $(R_D,R,\epsilon)$ are all trivially cohomology extremal, and, hence, {\rm (2)} is satisfied in this case.
Finally, notice that the hypothesis of {\rm (3)} is not satisfied when $Q\not = 0.$ 
Therefore, from now on, without loss of generality, we will assume that $Q=(0)$ so that $\aa= \bar \aa A$.

 (d) Let $D$ to be a distraction of $A$, $R=A/\aa=A/\bar \aa A,$ and $R_D=A/D(\aa)$;
 we are left to prove (2) and (3) for $R$. 
  We do so by proceeding as in the proof of Theorem \ref{abed}: we fix a homogeneous ideal $I$ of $R_D$,  denote by $J$ the pre-image of $I$ in $A$ and let ${\bf \omega}$ be the weight vector $(1,1,\ldots,1,0)$; furthermore, we fix a change of coordinates $g$ such that  $gD(X_{n})=X_{n}$, decompose the $\bar A$-module $\ini_{\bf \omega}(gJ)$ as $\ini_{\bf \omega}(gJ)=\bar J_{[0]} \oplus \bar J_{[1]}X_{n} \oplus \cdots \oplus \bar J_{[i]}X_{n}^i\oplus\cdots$ and observe that, by a standard upper-semi-continuity argument for Betti numbers and by \cite[Thm. 2.4]{Sb1},
\begin{equation*}\begin{split}
\beta_{ij}^A(R_D/I) \;=\; \beta_{ij}^A(A/J) \;\leq\;  \beta_{ij}^A(A/\ini_{\bf \omega}(gJ)) \hbox{\;\;\;\;\;\;\;\;\;\;\;\;\;\;\;\;\;\;for all $i, j$, and}\\
\Hilb \left(H^i_{\mm_A} (R_D/I)\right)_j \;=\; \Hilb \left(H^i_{\mm_A} (A/J)\right)_j \;\leq\; \Hilb
\left( H^i_{\mm_A}(A/\ini_{\bf \omega}(gJ) )\right)_j \hbox {\;\;\;\;\;\;for all\;} i, j .
\end{split}
\end{equation*}
Now, as in Remark \ref{udsre}, we denote by $\bar D$ the distraction of $\bar A$ obtained by considering the images in $\bar A$ of the first $n-1$ rows of $gD$, observe that $\bar D(\bar \aa)$ is the image, under the evaluation of $X_n$ at $0$, of $gD(\aa)$ in $\bar A$, and obtain that $\bar D(\bar \aa) \subseteq  \bar J_{[0]}$. Finally, $\bar D(\bar \aa)A \subseteq \ini_{\bf \omega}(gJ)$ and, hence,  $A/\ini_{\bf \omega}(gJ)$ is a quotient ring of  $R_{\bar D}:=A/\bar D(\bar \aa)A$. We denote by $I_{\bar D}$ the image of $\ini_\omega(gJ)$ in $R_{\bar D}$ and observe that $I$ and $I_{\bar D}$ have the same Hilbert function.

(e) By the inductive hypothesis there exists the embedding $\bar \epsilon _{\bar D}$ of $\bar R_{\bar D}:= \bar A / \bar D(\bar \aa)$ induced by 
 $\bar D$ and by the lex-embedding $\bar \epsilon$ of $\bar R$ (see (a)); moreover, both $(\bar R _{\bar D},\bar \epsilon_{\bar D})$ and $(\bar R_{\bar D}, \bar R, \bar \epsilon)$ are Betti and  cohomology extremal. Let $\epsilon_{\bar D}$ and $\epsilon$ denote the extensions of  embeddings $\bar \epsilon_{\bar D}$ and $\bar \epsilon$ to the rings  $\bar R_{\bar D}[X_n]= R_{\bar D}$ and $ \bar R[X_n]=R$ yielded by Theorem \ref{Exte}, respectively. Notice that $\epsilon$ is the lex-embedding, both $(R_{\bar D},\epsilon_{\bar D})$ and $(R,\epsilon)$ are Betti extremal, cohomology extremal, and by \cite[Theorem 3.1]{CaSb}, also $(R_{\bar D}, R, \epsilon)$ is cohomology extremal. 

(f) Since $\epsilon_{\bar D}$ is yielded by Theorem \ref{Exte},  we know that $\epsilon_{\bar D}(I_{\bar D})$ is the direct sum  $\bigoplus_d  \bar L_{[d]}X_n^d$ where each $\bar L_{[d]}$ is an embedded ideal of $\bar A / \bar D(\bar \aa)$ via $\bar \epsilon _{\bar D}.$ If we denote by $J_{\bar D}$ the pre-image of $\epsilon_{\bar D}(I_{\bar D})$ in $A$, we see that $J_{\bar D}$ can be written as
$\bigoplus_d  \bar J_{[d]}X_n^d$, where each $\bar J_{[d]}$ is the pre-image in $\bar A$ of $\bar L_{[d]}.$
Thus each $\bar J_{[d]}$ is the image, under $\bar D$, of a monomial ideal of $\bar A$ which is the sum of $\bar \aa$ and a lex-segment ideal of $\bar A.$ 
Viewing $\bar A$ as a subring of $A$, we may let $D'$ be the matrix obtained by adding to $\bar D$ a bottom row in which every entry is $X_n$; we observe that $D'$ is a distraction of $A$ and $J_{\bar D}=D'(U)$, for some monomial ideal $U$ such that $\bar\aa A=\aa\subseteq U$. Thus, $U$ has the same Hilbert function of $J$.

(g) We can now prove (3); we proceed as in the following diagram.\\

\vspace{-.5cm}
$$\begin{CD}\label{pathtoext} J\in\iP_A @>\ini_\omega\circ\,g>> \ini_\omega(gJ)\in\iP_A @. J_{\bar D}=D'(U)\in\iP_A @<D'<< U\in\iP_A @. L'\in\iP_A\\
@A \pi^{-1} AA  @VV \pi  V  @A\pi^{-1} AA @V\pi VV @A \pi^{-1} AA \\
I\in\iP_{R_D} @. I_{\bar D}\in\iP_{R_{\bar D}} @>\epsilon_{\bar D}>> \epsilon_{\bar D}(I_{\bar D})\in\iP_{R_{\bar D}} @. UR\in\iP_R @>\epsilon>> \epsilon(UR)\in\iP_R.
\end{CD}
$$
$$\text{Fig. 1.:\;\; $I$ has the  same Hilbert function as $\epsilon(UR)$}.$$\\

\vspace{-.5cm}

Since, by (e),  $(R_{\bar D},\epsilon_{\bar D})$ and $(R,\epsilon)$ are Betti extremal, and graded Betti numbers of monomial ideals do not change by applying a distraction - cf. Remark \ref{bedistra} -  we have, for all $i,j$ :
\[
\beta_{ij}^A(R_D/I) \;\leq\;\beta_{ij}^A(A/\ini_{\bf \omega}(gJ)) \;\leq\; \beta_{ij}^A(A/J_{\bar D}) \;=\; \beta_{ij}^A(A/U) \;\leq\; \beta_{ij}^A(R/\epsilon(UR))= \beta_{ij}^A(A/L'),
\]
where $L'$ be the pre-image of $\epsilon(UR)$ in $A$. The ideal $L'$ is the sum of $\aa$ and a lex-segment ideal of $A$ because $\epsilon$ is the lex-embedding of $R$ (see (e)). Since $I$ and $\epsilon(UR)$ have the same Hilbert function we have $\epsilon(I)=\epsilon(UR)$ and the above inequality implies that $(R_D,R,\epsilon)$ is Betti extremal. \\
Finally, since $\epsilon_D(I)= \epsilon_D(\epsilon(UR))$ is, by definition,   $D(L')R_D,$ we have
$\beta_{ij}^A(A/L') = \beta_{ij}^A(R_D/\epsilon_{D}(I))$ for all $i,j.$ Therefore, also $(R_D,\epsilon_D)$ is Betti extremal. 

(h) We conclude now by proving (2). We proceed as we did in (g); the difference from the previous case is that Hilbert functions of local cohomology modules are not preserved by applying a distraction but, by Proposition \ref{codistra}, we know that they cannot decrease.

Recall that, by (e), $(R_{\bar D}, \epsilon_{\bar D})$ and $(R_{\bar D}, R, \epsilon)$ are cohomology extremal, and  therefore, for all $i, j$,
 
$$\Hilb\left(H^i_{\mm_A}(A/\ini_{\bf \omega}(gJ))\right)_j  = \Hilb \left(H^i_{\mm_A}(R_{\bar D}/I_{\bar D}\right)_j  \leq  \Hilb \left(H^i_{\mm_A}(R/\epsilon(I_{\bar D}))\right)_j,$$ i.e. $(R_D,R,\epsilon)$ is cohomology extremal.
Since  $\epsilon(I_{\bar D})=\epsilon(I)$, an application of Proposition \ref{codistra} implies
$$ \Hilb \left(H^i_{\mm_A}(R/\epsilon(I_{\bar D}))\right)_j = \Hilb \left(H^i_{\mm_A}(A/L')\right)_j
\leq \Hilb \left(H^i_{\mm_A}(A/D(L'))\right)_j= \Hilb \left(H^i_{\mm_A}(R_D/\epsilon_D(I))\right)_j.$$ 
We have thus proven that $(R_D,\epsilon_D)$ is also cohomology extremal and completed the proof of the theorem.
\end{proof}


%


\end{document}